\documentclass[12pt]{article}

\usepackage[boxed] {algorithm2e}
\usepackage {amsmath}
\usepackage {amsfonts}
\usepackage {amsthm}
\usepackage {array}
\usepackage {comment}
\usepackage {enumerate}
\usepackage[mathscr] {eucal}
\usepackage {graphicx}
\usepackage {graphics}
\usepackage[utf8] {inputenc}
\usepackage {latexsym}
\usepackage {longtable}
\usepackage {xspace}
\usepackage {multirow}
\usepackage {url}

\usepackage{epstopdf}
\usepackage[caption=false]{subfig}
\usepackage[numbers,sort&compress]{natbib}
\bibpunct[, ]{[}{]}{,}{n}{,}{,}

\theoremstyle{plain}
\newtheorem{theorem}{Theorem}[section]

\theoremstyle{definition}

\theoremstyle{remark}

\DeclareMathOperator{\lin}{{\mathrm lin}}

\DeclareMathOperator{\co}{{\mathrm co}}
\DeclareMathOperator{\Co}{{\mathrm Co}}
\DeclareMathOperator{\Dim}{{\mathrm dim}}

\newcommand{\bleq}[2]{\begin{equation}\label{#1}{#2}\end{equation}}

\newcommand{\imset}{i=1,2,\dots, m}
\newcommand{\nll}{{\mathbf 0}}
\newcommand{\one}{{\mathbf 1}}
\newcommand{\bk}{b^K}
\newcommand{\bc}{b^{\co}}
\newcommand{\br}{b^\rho}

\newcommand{\rom}{\rho_{\min}}
\newcommand{\ac}{a^{\co}}
\newcommand{\R}{\mathbb{R}}

\newcommand{\octave}{OCTAVE\xspace}
\newcommand{\ptp}{PTP\xspace}
\newcommand{\qp}{QP\xspace}
\begin{document}
\title{Replacing projection on finitely generated convex cones with projection on bounded polytopes}
\author{
E.A. Nurminski \thanks{E.~A. Nurminski. Email: nurminskiy.ea@dvfu.ru} \\
Far Eastern Federal University, Vladivostok, Russia
}
\date{\today}
\maketitle

\begin{abstract}
This paper is devoted to the general problem of projection onto a polyhedral convex cone generated by a finite set of generators.
This problem is reformulated into projection onto the polytope obtained by simple truncation of the original cone. Then it can be solved with just two closely related projections onto the same bounded polytope.
This approach's computational performance is conditioned by
the crucial tool's efficiency for solving the fundamental problem of
finding the least norm element in a convex hull of a given finite set of points.
In our numerical experiments, we used for this purpose the specialized finite algorithm
implemented in the open-source system for matrix-vector calculations \octave.
This algorithm is practically indifferent to the proportions between the number of generators and their dimensionality
and significantly outperformed a general-purpose quadratic programming algorithm of the active-set variety
built into \octave when the number of points exceeds the dimensionality of the original problem.
\end{abstract}


\section*{Introduction}
This paper advocates a simple finite algorithm for solving orthogonal projection problems
onto closed convex cones formed by finite numbers of generators.
The numerous applications of this problem in restricted statistical inferences
\cite{resbook}, machine learning,
digital signal processing,
linear optimization
\cite{nurmi-globo}
and other problems
keep this problem as one of the most popular research subjects.
For instance,
it ranks 5th in popularity
in Wiki minimization internet resource \cite{wikimin}
with more than 480000 visits at the time of writing this note.

This close attention
inspired the development of many effective algorithms, especially for
isotone projection cones \cite{NN-isotone}.
For simplicial or lattice cones (see f.i. \cite{okk} and the references therein)
when the number of cone generators is less than the dimensionality
of underlying space, the results are less spectacular,
even when there are quite encouraging experiments
with different heuristics
\cite{AkNeNe2010}.

Here we consider the general case with the arbitrary number of generators 
which attracted less attention
possibly because of combinatorics,
involved in the selection of active face
on which the solution of the projection problem lies.
We are interested in a finite projection algorithm,
even if there are also iterative schemes for isotone cones \cite{iso}, which might be indispensable for huge-dimensional problems.

Great attention is given in the literature to projection onto truncated convex cones as local approximations of polyhedrons.
These truncated cones are then replaced by non-truncated simplicial cones
to simplify the following projection operations.
The simple idea of the current submission is just the opposite.
Instead, we suggest to project onto the finitely generated cone by truncating it to the bounded
polytope and use the well-established algorithm
\cite{ptp}
based on the article \cite{asm} with many computational improvements.
To justify this approach, we show the simple way to truncate a cone in such a manner that ensures the equivalence of truncated and non-truncated problems.
Then we present the polytope projection algorithm and compare its performance with the general-purpose quadratic programming code.
Finally, the polytope projection algorithm is applied to solve polyhedral cone projection problems known from the literature.
Apart from demonstrating the algorithm performance, it gave a chance to compare it with the popular iterative FISTA algorithm.
\section{Notations and preliminaries}
\label{NaP}
Our basic space is a finite-dimensional Euclidean space, denoted as \(E\).
If necessary, the dimensionality of \(E\) is derived as \(\Dim(E)\).
The non-negative ortant of \(E\) is denoted as \(E_+\).
The null vector of \(E\) is denoted as \(\nll\) to avoid confusion with the number \(0\).
We also use the notation \(\one = (1,1,\dots,1)\) to simplify certain expressions.
The real axis is denoted by \(\R\). The non-negative part of \(\R\) is denoted by \(\R_+\).

For \(\alpha\) in \(\R\) and \(A \subset E\) we use the shorthand \(\alpha A\) for \(\{ \alpha a, a \in A \}\).
Also we use \(A\times B\) as notation for \(\{ (a,b)\ a \in A, b \in B \}\).
Single-pont set \(\{a\}\) with \(a \in E\) is denoted simply as \(a\) when it
does not lead to misunderstanding.

Convexity of subsets of \(E\) is defined in a standard way.
The convex hull of set \(X\) denoted as \(\co\{X\}\) or \(X_c\)
and the conical hull of \(X\) is denoted as \(\Co\{X\}\).
When set \(X\) has the certian structure, namely \(X = X_1 \cup X_2 \cup ... \cup X_k\) we write \(\co\{X\} = \co\{X_1, X_2, ..., X_k\}\). 
We also denote as
\[
\lin\{X\} = \{ z = \sum_{i=1}^k \lambda_i x^i, ~ \lambda_i \in \R, x^i \in X, ~ i =1,2,\dots,k; ~ k = 1, 2, \dots \}
\]
the linear envelope of a set \(X\).
Of course, due to Karateodory \(k\) may be limited to \(\dim(E)\), if necessary.
The inequality \(x \geq y\) with \(x,y \in E\) is uderstood componentwise.

We denote the standard inner product in \(E\) as \(xy\),
the Euclidean norm is denoted \(|x\|^2 = xx\) as usual.
The unit ball in \(E\) is denoted as \( U_E = \{x : \|x\| \leq 1 \}\) or \(U\) if the space is clear from the context.
We also make some use of Chebyshev metric \(\|x\|_\infty = \max_{i = 1,2,\dots, \Dim(E)} \vert x_i\vert\)
and the correspondent unit ball 
\(U_\infty(E) = \{x: \|x\|_\infty \leq 1 \} \) or \(U_\infty \) if it clear enough.
The non-negative part of Chebyshev unit cube is denoted by 
\(U_\infty^+ = \{x: \|x\|_\infty \leq 1, x \geq \nll \} \).

The (orthogonal) projection problem in \(E\) for a point \(b\) and
a closed convex subset \(X\) of \(E\) is defined as 
\bleq{prj}{\min_{x \in X} \| x - b\|^2 = \| \Pi_X(b) - b \|^2.}
It defines the single-valued projection operator \(\Pi_X : E \to X \) which is well-defined
and has many useful properties such as Lipschitz continuity with unit Lipschitz constant,
non-expansion, and others.
Because of this, it is widely used in many computational algorithms and
theoretical considerations in convex analysis and beyond.

Our primal interest in this paper is the projection problem (\ref{prj}) when the set \(X\) is a closed convex cone,
so we introduce some additional relevant definitions.

A closed convex cone \(K\) is a closed subset of \(E\),
such that \(K = K + K\) and \(\alpha K = K\) for any \(\alpha > 0\).
We assume that further on, all cones are convex and closed.

For \(A \subset E\) a conical hull of \(A\) is defined and denoted as \(\Co\{A\} = \cup_{\lambda \geq 0} \lambda A_c\).
If \( K(A) = \Co\{A\}\), \(A = \{ a^i,\ \imset \}\) then \(K(A)\) is called finitely generated and \(a^i, \imset\) are called generators.
We call \(K(A)\) a {\em pointed} cone, if \( \nll \notin \co\{A\}\).

Considering the set \(A \subset E\) as an ordered collection of generators we can think of it as synonym for 
linear operator (matrix) \(A : E \to E' \)
which operates on \(E\) in the following way:
\[ Ax = y,\quad y = (y_1, y_2, \dots, y_m),\quad y_i = a^i x,\quad \imset, m = \Dim(E'). \]
In this notation a cone \(K(A)\) can also be defined as
\bleq{posi}{K(A) = A E'_+.}
Notice that we can also apply an arbitrary scaling \(K(A) = A D E'_+\) with
a diagonal matrix \(D\) with positive entries.

The projection problem
\bleq{conpro}{ \min_{ z \in K(A)} \| b - z \|^2 = \| b - \Pi_K(b) \|^2 = \|b - b^K\|^2 }
with solution \(b^K = \Pi_K(b)\)
for finitely generated cones with explicitly given generators \(A = \{ a^i, \imset \}\)
can be written as the semidefinite quadratic programming (QP) problem
\bleq{qp-form}{
\begin{array}[t]{c}
\min \| b - z \|^2 \\
z - \sum_{i=1}^m u_i a^i = \nll, \\
u_i \geq 0,\ \imset.
\end{array}
=
\begin{array}[t]{c}
\min \| b - z \|^2 \\
z - Au = \nll, \\
u \geq \nll
\end{array}
}
in \(n+m\) variables \(z, u\) and \(n + m\) linear constrains, where \(n = \Dim(E)\).
General-purpose QP solvers commonly require positive definiteness of objective and do not use this QP problem's specific form,
which leads to their disadvantage compared with specialized methods.
The problem (\ref{qp-form}) can be somewhat simplified by getting rid of the variable \(z\).
It leads to the equivalent problem
\bleq{U-problem}{
\begin{array}[t]{c}
\min \\
u \geq \nll
\end{array}
~ \{ uHu -2b A u \}
}
in baricentric coordinates \(u\) only with \(H = A^T A\), where \(A^T\) is the transpose of \(A\).

Each of these forms has its advantages-disadvantages:
(\ref{qp-form}) is semidefinite in the space of \(x,u\) variables, but preserve the sparseness of data,
(\ref{U-problem}) has a smaller number of variables, but the matrix \(H\) is often dense and can also be semidefinite if \(m > n\). 
\section{Scaled truncation}
To transform a projection problem for polyhedral cone into
the equivalent projection problem for the bounded polytope
we have to present a suitable way to transform the cone \(K(A)\) into the polytope \(Q(A)\)
in a way that makes (\ref{prj}) and (\ref{conpro}) equivalent, that is
\[
\min_{x \in K(A)} \| x - b \| = \| b^K - b \| = \min_{x \in Q(A)} \| x - b \| = \| b^Q - b \|, \quad b^K = b^Q.
\] 

The simple way to do this is to use the definition (\ref{posi})
and 
truncate the cone \(K(A)\) in the following way:
\bleq{trunbox}{
Q(A) = P_\rho(A) = \{ z = Au,\ u \in \rho U_\infty^+(E')\} = \rho AU_\infty^+(E') = \rho P(A)
}
where \(\rho > 0\) is a scaling parameter and \(P(A) = \{ z = Au,\ u \in U_\infty^+(E')\} = \co\{\nll, A\}\).

For what follows we introduce the following notations:
\begin{itemize}
\item
\(\bk\) is the solution of the projection problem
\bleq{alpha}{\min_{x \in K(A)} \| x - b \| = \|\bk - b\| = \alpha}
\item
\(\br\) is the solution of the projection problem
\bleq{beta}{\min_{x \in \rho A_c} \| x - b \| = \|\br - b\| = \beta}
\item
\(\bc\) is the solution of the projection problem
\bleq{gamma}{\min_{x \in \co\{\nll,\rho A_c\}} \| x - b \| = \|\bc - b\| = \gamma}
\item
\(\ac\) is the solution of the least-norm problem for the set \(A_c\)
\bleq{dist}{\min_{x \in A_c} \| x \| = \|\ac\|.}
Of course \(\min_{x \in \rho A_c} \| x \| = \rho \|\ac\|\).
\end{itemize}
We assume \(\alpha > 0, ~ \gamma > 0\) to avoid triviality.

Further on we define the affine function
\bleq{lmd}{\Lambda(x) = ( x - \br)(\br - b).}
Notice that \(\Lambda(\br) = 0\) and \(\Lambda(b) = -\beta^2 < 0 \) by construction and \(\Lambda(x) \geq 0\) for \(x \in \rho A_c\)
by optimality conditions in (\ref{beta}).

The following theorem establishes the simple condition for equivalence between truncated and untruncated projection problems.
\begin{theorem} 
If \(\rho > 0\) is such that \(\|\br\|^2 > b \br\) then \(\bk = \bc\).
\label{equilem}
\end{theorem}
\begin{proof}
The idea of the following proof is illustrated on the Fig. \ref{proof}.
By definition \(\alpha \leq \beta\).
If \(\alpha = \beta\) the theorem holds trivialy, therefore we assume the strict inequality \(\alpha < \beta\).
\begin{figure}
\begin{center}
\includegraphics[scale=0.7]{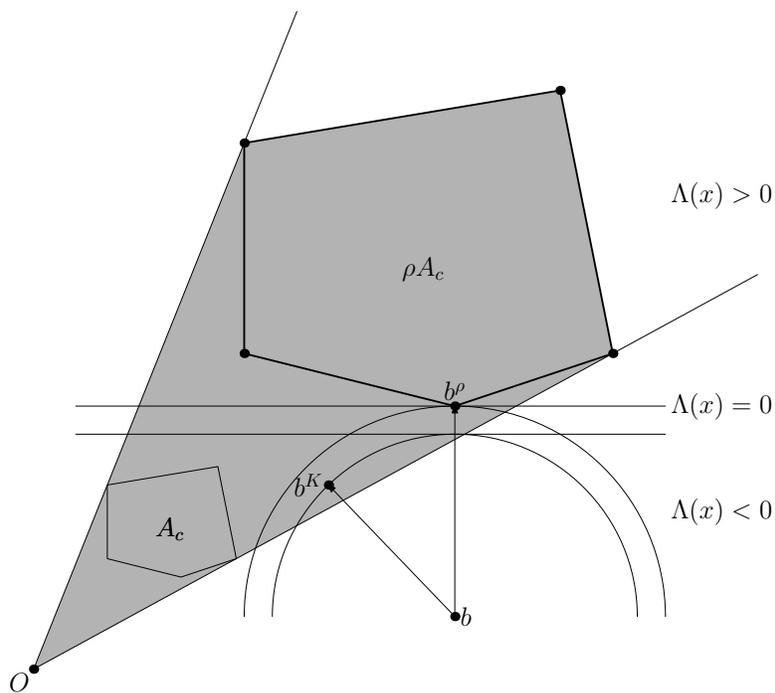}
\caption{The geometry of the proof. The shaded area is the truncated cone \(\co\{\nll, \rho A_c\}\)}
\label{proof}
\end{center}
\end{figure}
Then \(\bk \in b + \alpha U \subset b + \beta U\) where the inclusion is strict.
Estimate
\[
\begin{array}{c}
\Lambda(\bk) \leq 
\begin{array}[t]{c}
\max \\ x \in b + \alpha U
\end{array}
\Lambda(x) = \Lambda(b) +
\begin{array}[t]{c}
\max \\ z \in \alpha U
\end{array}
(\br - b)z = 
\\
\Lambda(b) + \alpha \|\br - b\| = -\beta^2 + \alpha\beta < 0.
\end{array}
\]
It implies \(\bk \in K(A) \cap \{x: \Lambda(x) \leq 0 \} \subset \co \{ \nll, \rho A_c \} \)
and by uniqness argument \(\bk = \bc\).
\end{proof}
Apart from its simplicity, the assumption of this theorem is easy to satisfy.
It is sufficient to solve the auxiliary least-norm problem \( \min_{x \in A_c} \|x\| = \|a^{\co}\| = d_A \)
and notice that \( \|\br\| \geq \rho d_A \) for any \(\rho > 0\).
By choosing \( \rho > \rom = \|b\|/d_A \) it can be obtained that \(\rho d_A > \|b\|\)
and therefore \(\|\br\| > \|b\|\) and hence
\[
\|\br\|^2 > \|b\|\|\br\| \geq b \br
\]
as theorem \ref{equilem} assumes.

It can be also shown directly that the condition \(\rho > \rom\),
which is more stringent than the condition of the theorem \ref{equilem},
is sufficient to garantee that \(b^k = \bc\). 
The geometry of this demonstration is shown on Fig. \ref{rmin}.
\begin{figure}
\begin{center}
\includegraphics[scale=0.7]{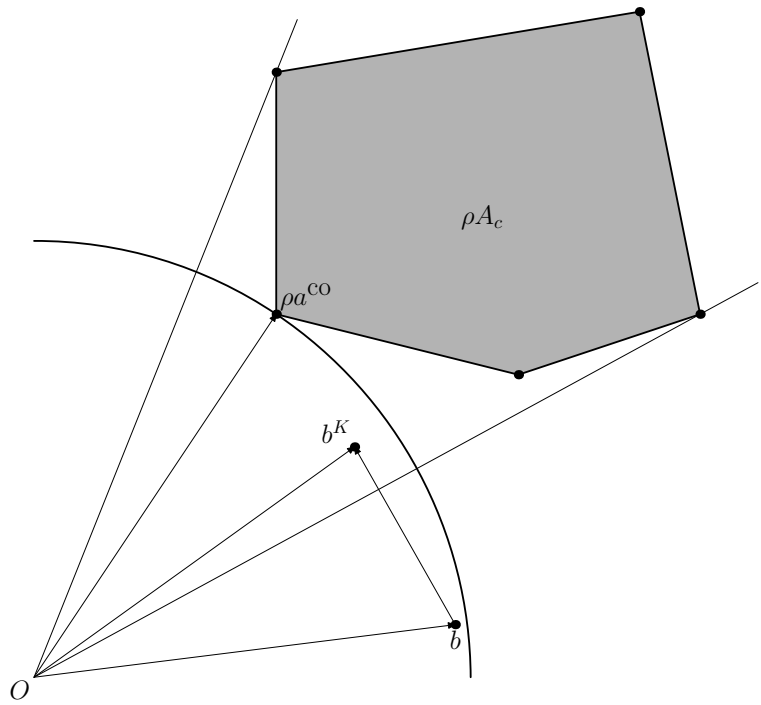}
\caption{The geometry of the proof for the condition \(\rho > \rom\).
The shaded area is the set \(\rho A_c\)}
\label{rmin}
\end{center}
\end{figure}
The formal proof goes as following: if \(\rho > \rom\) then
\[
\|x\|^2 \geq \rho^2 \|\ac\|^2 > \|b\|^2 = \|\bk\|^2 + \| \bk - b \|^2 \geq \|\bk\|^2
\]
for any \(x \in \rho A_c\) by optimality condition of \(\bk\).

Hence \(\bk \notin \rho A_c\) and moreover
\[
\|\bk\| < \theta \|\rho \ac \| \leq \theta \|x\|
\]
for \(\theta \geq 1\) and any \(x \in \rho A_c\).

Therefore \(\bk \notin \{Au, u \geq \one \}\) or \(\bk \in \co\{\nll, \rho A_c \}.\)
Then
\[
\|\bk - b\| = \min_{x\in K(A)} \|x- b\| \leq \min_{x \in \co\{\nll, \rho A_c\}} \| x - b \| \leq \|\bk - b\|
\]
and \(b^K = \bc\) which terminates this demonstration.

The following Algorithm is based on this truncation.
Using the estimate \(\rho > \rho_{\mathrm min} = \|b\|/\| b^{\mathrm min}\| \)
we can solve cone projection problem 
(\ref{conpro}) by solving 2 closely related projection problems
specified in Algorithm \ref{algo}.

\begin{algorithm}
\KwData{
The set \( A = \{ a^i, \imset \} \) of generators of a cone \(K(A)\),
the vector \(b\) to be projected on the cone \(K(A)\).
}
\KwResult{ The solution vector \(b^K\) of the problem (\ref{conpro}).}
{\bf Phase 1.}
Compute a suitable value for the scaling parameter \(\rho\) by solving
the auxiliary polytope projection problem
\bleq{phase-1}{ \min_{z \in A_c} \| z \|^2 = \| b^{\mathrm min} \|^2, }
and setting an appropriate lower estimate for the scaling parameter as \( \rho_{\mathrm min} = \|b\|/\| b^{\mathrm min} \| \).
\par\noindent
{\bf Phase 2.}
Choose any \(\rho \geq \rho_{\mathrm min} \) and solve the projection problem 
\bleq{phase-2}{\min_{z \in \co\{\nll, \rho A_c \}} \| z - b \|^2 = \| \br - b \|^2 = \|b^K - b\|^2, }
according to Lemma \ref{equilem}.
\caption{The Cone Truncated to Polytope (CTP) algorithm}
\label{algo}
\end{algorithm}
Notice that (\ref{phase-1}) and (\ref{phase-2}) are both projections on polytopes
which differ by one vertex only so
we can use advanced feasible basis in one problem to start the solution process
for the other which make this approach even more attractive.

\section{Numerical experiments}
This section describes our experience with the suggested approach
to solve projection problems for polyhedral cones by reducing them to polytope projections.
For this approach to be useful in a practical sense, we have to be first of all sure
that the polytope projection problems are solved efficiently for polytopes
of different proportions.
We use in our numerical experiments the in-house \octave \cite{octave} implementation of the algorithm
\cite{RG-proj}
which we consider now sufficiently reliable and efficient.
We substantiate this claim by computational experiments which compare
the performance of this routine with of-the-shelf \octave 
general-purpose quadratic optimization function \qp of active-set variety.

These experiments are conducted with the test problems, inspired by P.Wolfe's "pancake" random sets
\cite{Wolfe-cake}.
These sets are determined by two scalar parameters \(0 < \delta < \Delta\)
and consist of, say, \(m\) random points uniformly distributed
in a \(n\)-dimensional box-like set
\(B_{\delta, \Delta}\) which is a composition of \((n-1)\)--dimensional cube, scaled by \(\Delta\)
\(\Delta [-1, +1]^{n-1}\) and the interval \([-1, +1]\) scaled by \(\delta\)
and shifted up by \(2\delta\).

The simple expression, where \(U^k_\infty\) is the \(k\)--dimensional Chebyshev unit ball
\bleq{wolfe-pance}{ B_{\delta, \Delta} = \Delta U^{n-1}_\infty \times \delta (U^1_\infty + 2) \subset E^n }
can formally describe this set and Fig. \ref{wpc} demonstrates the geometry of this data-set for 3-dimensional case.
\begin{figure}
\begin{center}
\includegraphics{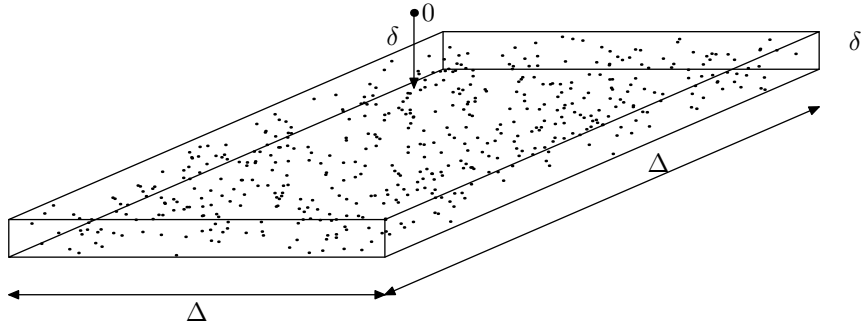} \\
\caption{The geometry of the dataset (\ref{wolfe-pance})}
\label{wpc}
\end{center}
\end{figure}
The least-norm problem for such sets can be made computationally rather difficult either by choosing extreme values
of \(\delta, \Delta\) and/or by populating \(B_{\delta, \Delta}\) with a large number of points.
If \(m \gg n\) and \(n\) is sufficiently large
the random points tend to distribute along the sides of \(B_{\delta, \Delta}\)
which makes the least-norm problem computationally ill-conditioned for \(\delta \ll \Delta\).

The following experiments, described in subsection \ref{cones} were conducted with the polyhedral cones,
generated by the sets of generators, produced by the same expression (\ref{wolfe-pance}).
The special attention in these massive experiments was given to the peculiarities in solving
projection problems (\ref{phase-1}, \ref{phase-2}) for different total sizes of the data-sets.
Here again, the polytope projection routine demonstrated its reliability
and effectiveness. 
\subsection{Specialized versus General-Purpose}
\label{polies}
The tables \ref{m-ptp-tests},
\ref{n-ptp-tests}
demonstrate the results of tests
with polytopes of different proportions
between numbers of points in the bundles and
the dimensionality of the underlying space.
The common notations in these tables:
{\tt ptp} --- the solution time (sec) for the \ptp routine,
{\tt qp} --- the same for the \qp function,
{\tt obj-ptp, obj-qp} --- objective values for \ptp and \qp
respectively, {\tt \ptp/\qp} --- the time-ratio for these algorithms.
\begin{table}
\begin{center}
\caption{The test with the fixed number of points \(m = 600\) in the bundle, \(n\) --- dimensionality of the space. }
\label{m-ptp-tests}
\begin{tabular}{|rrrrrr|}
\hline
\multicolumn{1}{|c}{n} &
\multicolumn{1}{c}{\tt ptp} &
\multicolumn{1}{c}{\tt qp} &
\multicolumn{1}{c}{obj-ptp} &
\multicolumn{1}{c}{obj-qp} &
\multicolumn{1}{c|}{ptp/qp} \\
\hline
100 & 1.03 & 408.58 & 6.18602490e-03 & 6.18602490e-03 & 0.003 \\
200 & 2.42 & 420.37 & 7.42018810e-03 & 7.42018810e-03 & 0.006 \\
300 & 4.86 & 410.15 & 8.75037778e-03 & 8.75037778e-03 & 0.012 \\
400 & 8.11 & 267.24 & 9.51469830e-03 & 9.51469830e-03 & 0.030 \\
500 & 9.42 & 163.81 & 1.03756037e-02 & 1.03756037e-02 & 0.058 \\
600 & 5.93 & 2.78 & 1.09087458e-02 & 1.09087458e-02 & 2.136 \\
700 & 5.79 & 0.55 & 1.13905009e-02 & 1.13905009e-02 & 10.585 \\
800 & 5.94 & 0.55 & 1.15418384e-02 & 1.15418384e-02 & 10.811 \\
900 & 6.10 & 0.56 & 1.11699504e-02 & 1.11699504e-02 & 10.840 \\
1000 & 6.21 & 0.56 & 1.12752146e-02 & 1.12752146e-02 & 11.057 \\
\hline
\end{tabular}
\end{center}
\end{table}
\begin{table}
\begin{center}
\caption{ Test with the fixed dimensionality \(n = 600\) of points in the bundle, \(m\) --- the number of points.}
\label{n-ptp-tests}
\begin{tabular}{|rrrrrr|}
\hline
\multicolumn{1}{|c}{m} &
\multicolumn{1}{c}{\tt ptp} &
\multicolumn{1}{c}{\tt qp} &
\multicolumn{1}{c}{obj-ptp} &
\multicolumn{1}{c}{obj-qp} &
\multicolumn{1}{c|}{ptp/qp} \\
\hline
100 & 0.29 & 0.17 & 1.10150924e-02 & 1.10150924e-02 & 1.699 \\
200 & 0.68 & 0.06 & 1.09582666e-02 & 1.09582666e-02 & 12.369 \\
300 & 0.98 & 0.13 & 1.13598866e-02 & 1.13598866e-02 & 7.553 \\
400 & 1.92 & 0.23 & 1.10932105e-02 & 1.10932105e-02 & 8.525 \\
500 & 3.32 & 0.37 & 1.13554032e-02 & 1.13554032e-02 & 8.889 \\
600 & 6.04 & 3.10 & 1.09087458e-02 & 1.09087458e-02 & 1.946 \\
700 & 12.43 & 295.61 & 1.02334609e-02 & 1.02334609e-02 & 0.042 \\
800 & 19.23 & 657.93 & 9.91185272e-03 & 9.91185272e-03 & 0.029 \\
900 & 25.13 & 1394.20 & 9.36404467e-03 & 9.36404467e-03 & 0.018 \\
1000 & 32.36 & 2050.40 & 9.03549357e-03 & 9.03549357e-03 & 0.016 \\
\hline
\end{tabular}
\end{center}
\end{table}
First of all
it can be concluded from these exeriments that \ptp routine is more robust
than \qp solver.
Solution time for \ptp grows approximately linear in these cases,
when for \qp solver it radically increases when the number of points becomes larger
than their dimension.
Highly likely that it is the result of possible semi-definiteness of the quadratic form in (\ref{U-problem}).
The \ptp routine is slower on simple problems when \(m < n\), but even in these cases, its performance is acceptable, and solution time is still low.

It can be concluded from these experiments that \ptp might be a function of choice for the cases when the number of points
in a dataset is either greater than their dimensionality or varies on a large scale.
\subsection{Numerical experiments with the cones}
\label{cones}
We performed three series of computational experiments
to demonstrate peculiarities of computational efficiency of the CTP algorithm.
Each series consisted in several experiments with conic projection problems
with dimensionality \(m\) and number of vectors \(n\) such that the size \(m \cdot n\) of the matrix \(A\)
is approximately the same, up to the integrality of \(m\) and \(n\).
These series were determined by three values of the product \(m \cdot n = 10^6, 2\cdot 10^6\) and \(3\cdot 10^6\)
and
in each of these series the dimensionality \(m\) was increasing by 50 starting with the initial value 50 while
there was a corresponding \(n \geq 500\) such that the product \(m \cdot n\) does not exceed and
is as close as possible to \(10^6, 2\cdot 10^6\) or \(3\cdot 10^6\).
All in all there were solved 237 such problems with \(\delta = 0.01\) and \(\Delta = 50\).
\begin{figure}
\begin{center}
\begin{minipage}{.45\textwidth}
\begin{center}
\includegraphics[scale=0.5]{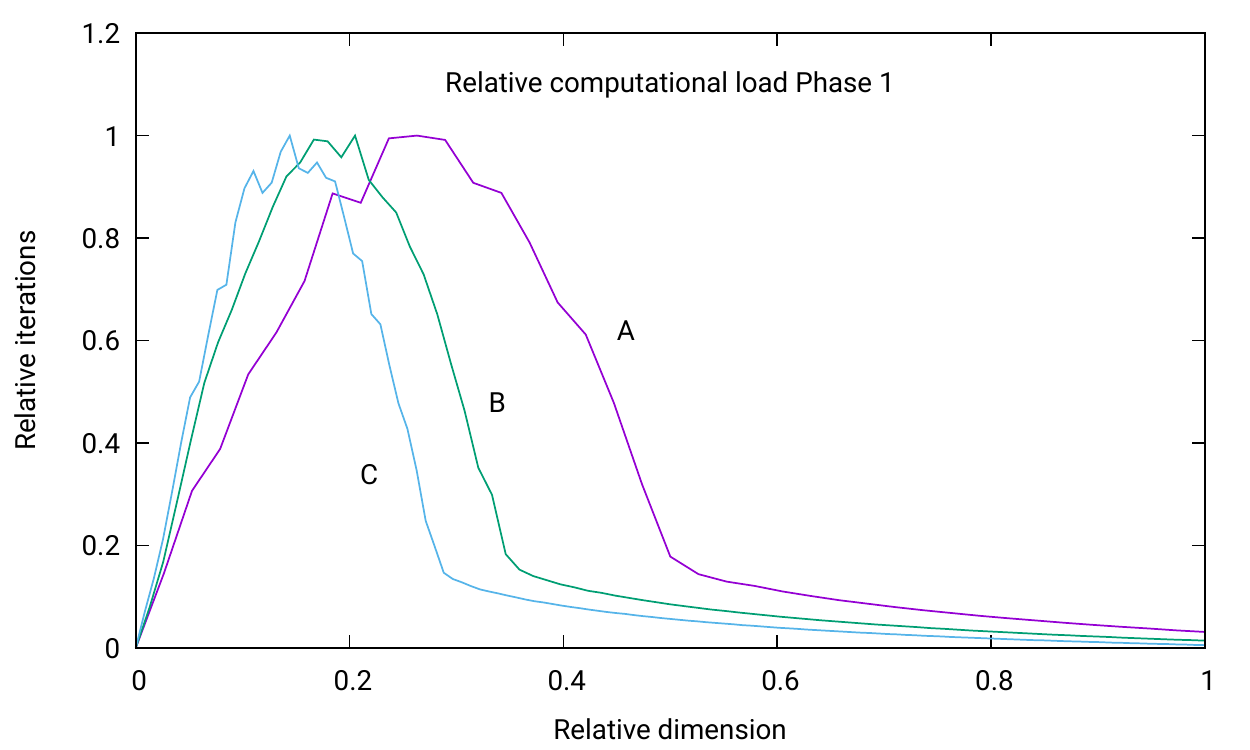} \\
Solution of the problem (\ref{phase-1})
\end{center}
\end{minipage}
\begin{minipage}{.45\textwidth}
\begin{center}
\includegraphics[scale=0.5]{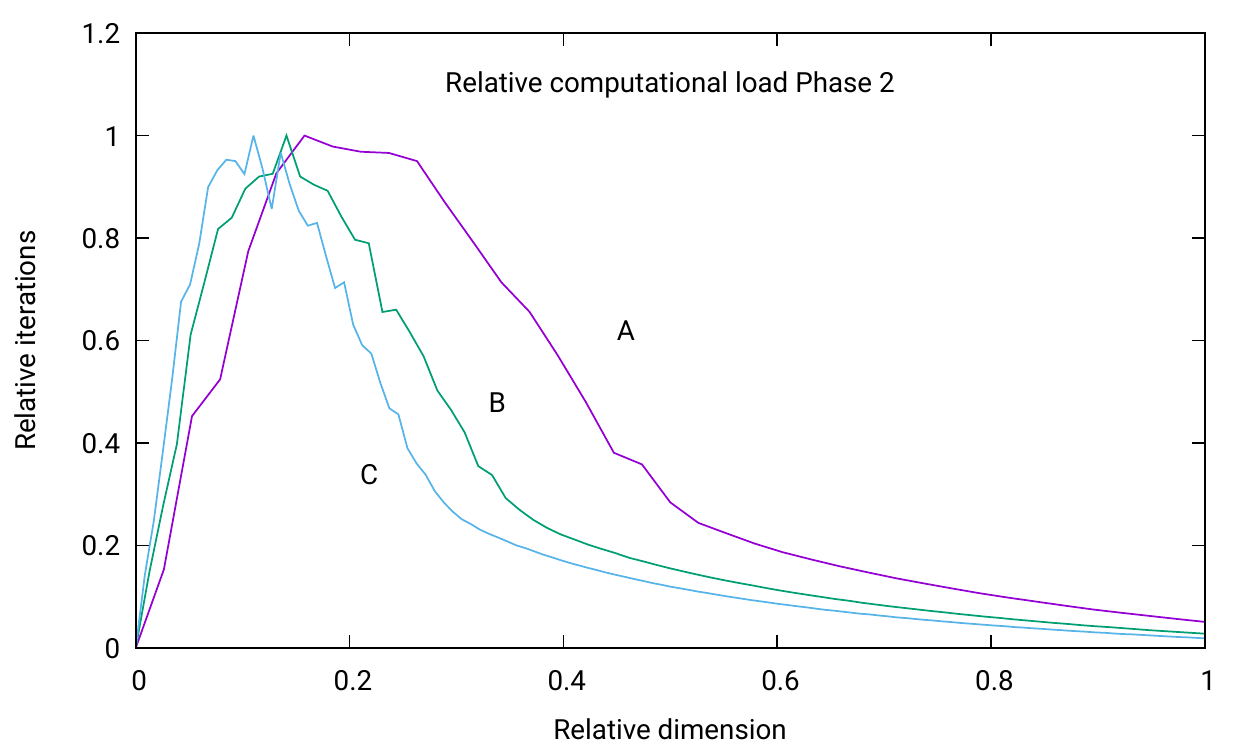} \\
Solution of the problem (\ref{phase-2}).
\end{center}
\end{minipage}
\\
\caption{
Computational complexity for
solutions of problems
(\ref{phase-1}, \ref{phase-2}
for three different values of the size of the data set:
A -- \(10^6\),
B -- \(2 \cdot 10^6\) and
C -- \(3 \cdot 10^6\) dual precision elements.
}
\end{center}
\label{fig1}
\end{figure}
The computational complexity of a solution for every problem instance
was measured in terms of the number of iterations of
underlying projection subroutines.
The results of these experiments are shown in the aggregated form in Fig. \ref{fig1}.
The iteration complexity of the CTP algorithm is demonstrated as the function of the relative dimensionality of the underlying space of variables.
The projection iterations and dimensionality were scaled to \([0,1]\) intervals to make them comparable.
The actual ranges of these values are demonstrated in Table \ref{rnges}.

The interesting feature of these graphs is that in all three series they demonstrated the highest complexity
somewhere close to each other in the range of lower dimensional problems with larger number of generators,
however for smallest dimensions with highest number of points the complexity was decreasing again.
\begin{table}
\begin{center}
\caption{Maximal/minimal values for dimensionality (\(m\)), number of generators (\(n\)), iterations (\(it\)) for 2 phases of the algorithm}
\begin{tabular}{crrr|rrr|}
& \multicolumn{6}{c}{} \\\cline{2-7}
\multirow{2}{10mm}{} & \multicolumn{3}{|c}{phase-1} & \multicolumn{3}{|c|}{phase-2} \\\cline{2-7}
& \multicolumn{1}{|c}{\(m\)} & \multicolumn{1}{c}{\(n\)} & \multicolumn{1}{c|}{\(it\)} & \multicolumn{1}{c}{\(m\)} & \multicolumn{1}{c}{\(n\)} & \multicolumn{1}{c|}{\(it\)} \\
\hline
\multicolumn{1}{|c|}{A} & 2950/50 & 20000/513 & 4557/382 & 1950/50 & 20000/513 & 2749/392 \\
\hline
\multicolumn{1}{|c|}{B} & 3950/50 & 40000/506 & 7366/407 & 3950/50 & 40000/506 & 4263/398 \\
\hline
\multicolumn{1}{|c|}{C} & 5950/50 & 60000/504 & 10138/450 & 5950/50 & 60000/504 & 5314/411 \\
\hline
\end{tabular}
\end{center}
\label{rnges}
\end{table}
The \octave-script which implements this algorithm can be found at ResearchGate by following the reference \cite{ocode}.
\subsection{Quadratic optimization}
To show the difference in performance of the finite algorithm given in this paper and the popular iterative algorithm FISTA
\cite{bt09}
we applied our algorithm to solve the quadratic optimization problem described there as one of the tests.
The latter is the trivial quadratic \(n\)-dimensional problem
\bleq{quad}{ \min_{x\in E^n} ~ \|Ax\|^2 = \min_{z \in \lin(A)} ~ \|z\|^2,}
where \(A\) is the Laplacian operator
\[ A = 
\begin{bmatrix}
2 & -1 & & & \\
-1 & 2 & -1 & & \\
& & \dotsm & & \\
& & -1 & 2 & -1 \\
& & & -1 & 2 \\
\end{bmatrix}_{n} .
\]
with \(n = 201\) and \(\lin(A)\) is the linear envelope of the matrix \(A\), considered as the set of its columns.

To convert the problem (\ref{quad}) into conical projection,
notice that \(\lin(A) = \Co\{ -b, A \} = K_A\),
where \(b\) is such vector that \(A b \geq \one\).
As the columns of \(A\) are linear independent, such vector exists and can be obtained as a solution of the least-norm problem
\bleq{squad}{\min_{z \in \co\{A\}} ~ \|z\|^2 = \|b\|^2.}
After this the problem (\ref{quad}) is solved by the Algorithm \ref{algo}
by projecting \(\nll\) on the cone \(K_A\).
Due to the particular nature of this problem, it is not a great surprise that both these problem were
solved in \(n\) and \(n+1\) iterations.
The details of these two solution processes are shown on Fig. \ref{fig2}.
\begin{figure}
\begin{center}
\begin{minipage}{.45\textwidth}
\begin{center}
\includegraphics[scale=0.5]{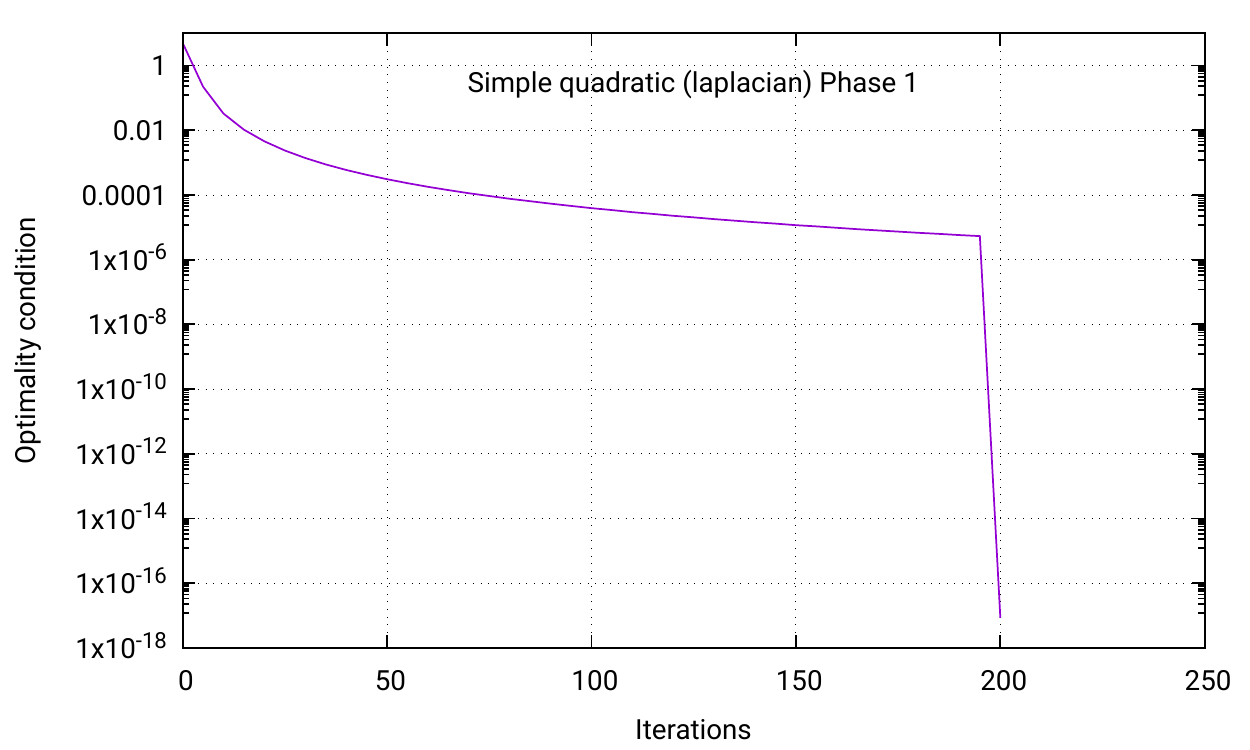} \\
Solution of the problem (\ref{squad})
\end{center}
\end{minipage}
\begin{minipage}{.45\textwidth}
\begin{center}
\includegraphics[scale=0.5]{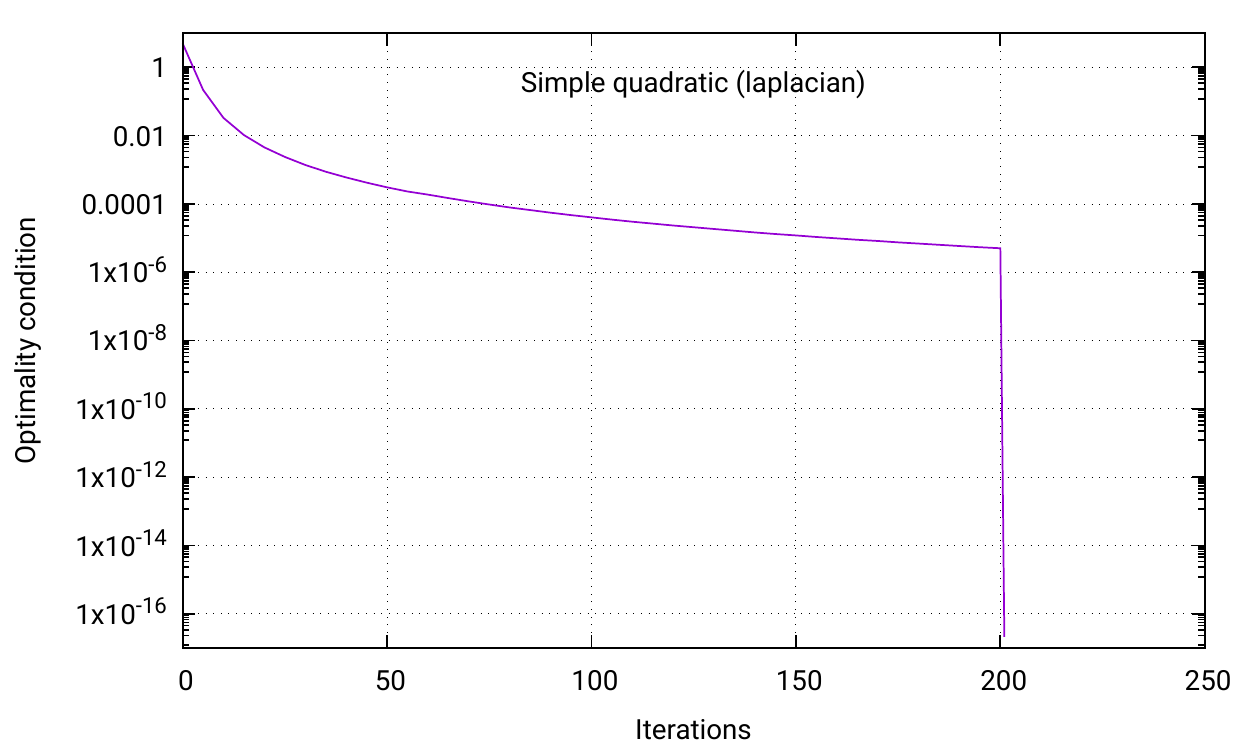} \\
Solution of the problem (\ref{quad}).
\end{center}
\end{minipage}
\\
\caption{
Computational process for solutions of problems (\ref{squad}, \ref{quad}).
Optimality conditions.
}
\end{center}
\label{fig2}
\end{figure}
It is interesting to notice the great similarity between these two processes.
It can be said that both of them, for the most iterations, mainly collected information
on the structure of the set of generators without a substantial decrease in violation of optimality conditions.
When the algorithm collected this information, the projections are computed at the single final step of the processes
practically up to machine accuracy.

For comparison, we present in Fig. \ref{fistas} the results from \cite{imfista} relating to the same quadratic problem.
These graphs show the convergence of the whole family of FISTA-type algorithms.
The convergence of these algorithms is demonstrated there from the point of view of two criteria:
the distance \(\|x^k - x^\star\|\) of the \(k\)-th iteration \(x^k\) from the solution \(x^\star = \nll\)
and deviation of the value of objective function \(\Phi(x^k) = \|x^k\|^2\) from the optimum \(\Phi(x^\star) = \Phi(\nll) = 0\).
One can see the significant progress in acceleration of these algorithms.

\begin{figure}
\subfloat[\(\|x^k- x^\star\|\)]
{ \includegraphics[width=0.4\linewidth]{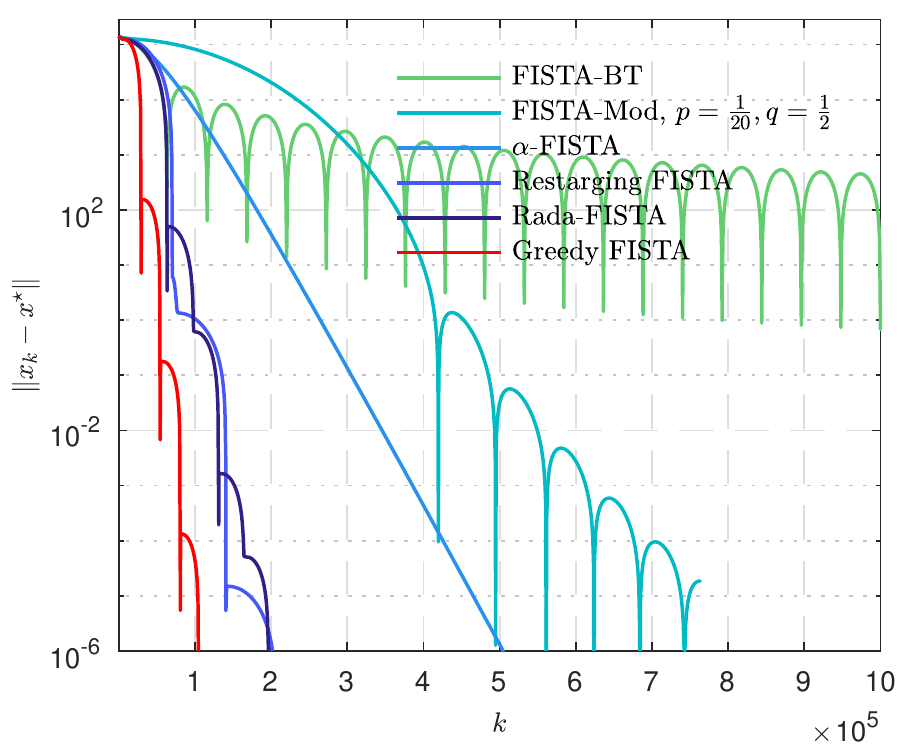} }
\hspace{24pt}
\subfloat[\(\Phi(x^k)-\Phi(x^\star)\)]
{ \includegraphics[width=0.4\linewidth]{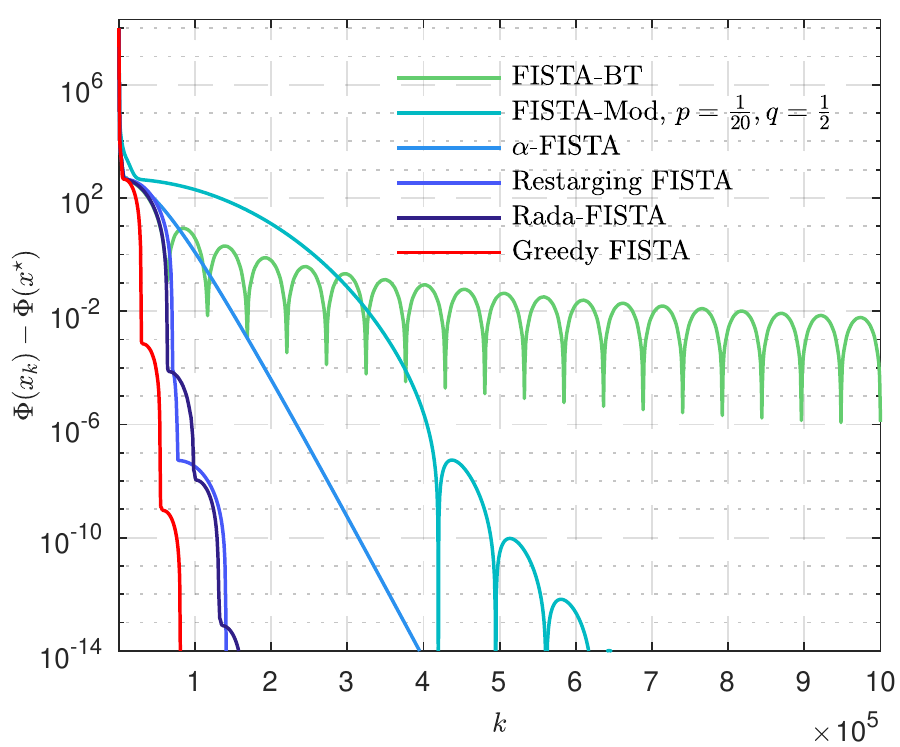} }
\\
\caption{Comparison of different FISTA schemes for least square problem (\ref{quad}).
(a) convergence of \(\|x^k-x^\star\|\); (b) convergence of \(\Phi(x^k)-\Phi(x^\star)\).}
\label{fistas}
\end{figure}

Unfortunately, there are no run-time measurements for these codes in \cite{imfista},
so there is no direct way to compare the practical effectiveness of these algorithms with the CTP routine.
It can only be said that even the fastest of the FISTA-algorithms
by the number of iterations is about 3-orders of magnitude slower than the finite algorithm presented here.
Of course, FISTA-algorithms have a wider applicability area and smaller memory footprint, but it looks like that for
quadratic problems of the medium size, the exact finite algorithm is preferable.
\section*{Conclusions}
As a conclusion, it may be stated that the result of this notice
paves the way for robust and practical approaches for solving 
projection problems on polyhedral cones independently of the proportions between their dimensionality and number of generators.
The computational efficiency of the described algorithm
can be improved
by providing intelligent restart procedures for both phases
of the algorithm or by other means, which will be a subject of further investigations.

\end{document}